%% file: dyn.tex
\documentclass{amsart}
\usepackage{amssymb}
\usepackage{amsthm}
\usepackage{epsfig}
\usepackage{color}
\usepackage{hyperref}
\usepackage{mathtools}

\newcommand{\mc}[1]{\mathcal{#1}}
\newcommand{\stab}{\mathrm{Stab}}
\newcommand{\bH}{\mathbb{H}}
\newcommand{\cd}{\mathrm{cd}}
\newcommand{\aicd}{\mathrm{ai\check{C}d}}

\def\co{\colon\thinspace}
\def\R{\mathbb{R}}
\def\Z{\mathbb{Z}}

\def\N{\mathbb{N}}

\newtheorem{theorem}{Theorem}[section]
\newtheorem{lemma}[theorem]{Lemma}
\newtheorem{corollary}[theorem]{Corollary}
\newtheorem{proposition}[theorem]{Proposition}

\newtheorem{conjecture}[theorem]{Conjecture}

\newtheorem{claim}{Claim}

\theoremstyle{definition}
\newtheorem{remark}[theorem]{Remark}
\newtheorem{definition}[theorem]{Definition}

\begin{document}

\title{The Bowditch boundary of $(G,\mc{H})$ when $G$ is hyperbolic}
\author{Jason Fox Manning}

\begin{abstract}
In this note we use Yaman's dynamical characterization of relative hyperbolicity to prove a theorem of Bowditch about relatively hyperbolic pairs $(G,\mc{H})$ with $G$ hyperbolic.  Our proof additionally gives a description of the Bowditch boundary of such a pair.  This description of the boundary was previously obtained by Tran \cite{Tran13}.
\end{abstract} 

\thanks{The support of the Simons Foundation (\#524176 to J. Manning) and the National Science Foundation (DMS-0804369) is gratefully acknowledged.}
\maketitle

\section{Introduction}
Let $G$ be a group.  A collection $\mc{H} = \{H_1,\ldots,H_n\}$ of
subgroups of $G$ is said to be \emph{almost malnormal} if every
infinite intersection of the form
$H_i\cap g^{-1} H_j g$ satisfies both $i=j$ and
$g\in H_i$.

In an extremely influential paper from 1999, published in 2012 in \emph{IJAC} \cite{bowditch:relhyp}, Bowditch proves the following useful theorem:

\begin{theorem}\cite[Theorem 7.11]{bowditch:relhyp}\label{t:bowditch}
Let $G$ be a nonelementary hyperbolic group, and let $\mc{H} = \{H_1,\ldots,H_n\}$
be an almost malnormal collection of proper, quasiconvex subgroups of
$G$.  Then $G$ is hyperbolic relative to $\mc{H}$.
\end{theorem}
\begin{remark}
  The converse to this theorem also holds and is implicit in Bowditch's work.  If $(G,\mc{H})$ is any
  relatively hyperbolic pair, then the collection $\mc{H}$ is almost
  malnormal by \cite[Proposition 2.36]{osin:relhypbook} (cf. \cite[p. 4]{bowditch:relhyp}).  Moreover the
  elements of $\mc{H}$ are undistorted in $G$ \cite[Lemma 5.4]{osin:relhypbook} (cf. \cite[Lemma 3.5]{bowditch:relhyp}).  
  Undistorted subgroups of a hyperbolic group are quasiconvex.
\end{remark}

In this note, we give
a proof of Theorem \ref{t:bowditch}
which differs from Bowditch's.  The strategy we follow is
to exploit the dynamical characterization of relative hyperbolicity
given by Yaman in \cite{yaman04}.   By doing so, we are able to obtain
some more information about the pair $(G,\mc{H})$.  In particular, we
obtain an explicit description of its Bowditch boundary
$\partial(G,\mc{H})$.  (This same strategy was applied by Dahmani to describe the boundary of certain amalgams of relatively hyperbolic groups in \cite{Dahmani03}.)
Let $\partial G$ be the Gromov boundary of the group $G$.
If $H$ is quasiconvex in a hyperbolic group
$G$, its limit set $\Lambda (H) \subset \partial G$ is homeomorphic to
the Gromov boundary $\partial H$ of $H$.  Our proof of Theorem \ref{t:bowditch} also yields the following result (previously obtained by Tran \cite{Tran13}), which says that $\partial(G,\mc{H})$ is obtained by smashing the limit sets of $gHg^{-1}$ to points, for $H\in \mc{H}$ and $g\in G$.  
\begin{theorem}\label{t:decomp} 
  Let $G$ be hyperbolic,  and let
  $\mc{H}$ be an almost malnormal collection of infinite quasi-convex proper subgroups
  of $G$.  Let $\mathcal{L}$ be the set of
  $G$--translates of limit sets of elements of $\mc{H}$.
  The Bowditch boundary
  $\partial(G,\mc{H})$ is obtained from the Gromov boundary $\partial
  G$ as a decomposition space $\partial G/ \mathcal{L}$.
\end{theorem}

\begin{remark}\label{rk:history}
  After I posted a version of this paper on the arXiv, I learned that Theorem \ref{t:decomp} was already well-known.  See in particular the main result of Tran's paper \cite{Tran13} which additionally gives a similar description of the Bowditch boundary in terms of a CAT$(0)$ boundary when $G$ is CAT$(0)$ and relatively hyperbolic.  Tran also points out previous results of Gerasimov and Gerasimov--Potyagailo \cite{Gerasimov12,GerasimovPotyagailo13}, or alternatively Matsuda--Oguni--Yamagata \cite{Matsuda12} which can be used to give other proofs of Theorem \ref{t:decomp}.  More recently, an ``HHS'' proof can be found in \cite[Section 6]{Spriano}.

  If there is an advantage to the current approach, it is that we obtain a proof of both Theorems \ref{t:bowditch} and \ref{t:decomp} at the same time.
\end{remark}

One consequence of the explicit description is a bound on the dimension of such a Bowditch boundary.
\begin{corollary}\label{dimcor}
  Let $G$ be a hyperbolic group and $\mc{H}$ an almost malnormal collection of
  infinite quasi-convex proper subgroups.  Then $\dim \partial
  (G,\mc{H})\leq \dim \partial  G + 1$.
\end{corollary}
\begin{proof}
  This follows from the Subspace and Addition Theorems of dimension theory.
  By Theorem
  \ref{t:decomp}, $\partial(G,\mc{H})$ can be written as a 
  union of a countable set $A$ (coming from the limit sets of the
  conjugates of the elements of $\mc{H}$) with a subspace $B$ of $\partial
  G$.  The Subspace Theorem implies $\dim(B)\leq \dim\partial G$, and the Addition Theorem implies $\dim(A\cup B)\le \dim A + \dim B + 1$.
\end{proof}

We can see this as some weak evidence for the following conjecture.  (Here $\cd(G,\mc{H})$ is the maximum $n$ so $H^n(G,\mc{H};M)\neq 0$ for some $\Z G$--module $M$.)
\begin{conjecture}\cite{me:MW1}\label{conj:dimeq}
  Let $(G,\mc{H})$ be relatively hyperbolic and type $F$.  Then
  \begin{equation*}
    \dim\partial(G,\mc{H})=\mathrm{cd}(G,\mc{H})-1.
  \end{equation*}
\end{conjecture}
In the absolute setting ($\mc{H} = \emptyset$) Conjecture \ref{conj:dimeq} is a theorem of Bestvina and Mess \cite{bestvinamess}.  It is shown in \cite{me:MW1} that it also holds in case $\cd(G)<\cd(G,\mc{H})$.

    Here is the connection between Corollary \ref{dimcor} and the conjecture.
In \cite{me:MW1}  it is shown that if $(G,\mc{H})$ is relatively hyperbolic and type $F_\infty$, then for all $k\geq 0$ there is an isomorphism
\begin{equation}\label{mwiso}\tag{$*$}
  \check{H}^k(\partial(G,\mc{H});\Z) \cong H^{k+1}(G,\mc{H};\Z G),
\end{equation}
where the left-hand side is reduced \v{C}ech cohomology and the right-hand side is relative group cohomology as defined for example in \cite{BE78}.\footnote{In the absolute case, this is another theorem of Bestvina--Mess \cite{bestvinamess}.  In the case of geometrically finite groups of isometries of $\bH^n$ it is due to Kapovich \cite[Proposition 9.6]{KapovichMR2491697}.}
It follows that the inequality $\dim\partial(G,\mc{H})\geq\mathrm{cd}(G,\mc{H})-1$ always holds for a type $F$ pair, since for any space $X$ we have the inequality
\[ \dim(X)\geq \max\{k\mid \check{H}^k(X;\Z)\neq 0 \}\eqqcolon\aicd(X).\]
(The notation `$\aicd$' stands for \emph{absolute integral \v{C}ech dimension}.)
The statements \cite[Corollaries 1.3(b) and 1.4(b)]{bestvinamess} combined show that, for a hyperbolic group, $\aicd(\partial G)=\dim(\partial G)$.
The isomorphisms \eqref{mwiso} and the long exact sequence of a group pair then give
\begin{equation}\label{aicdineq}\tag{$\dagger$} \aicd\left(\partial(G,\mc{H})\right)\leq \aicd(\partial G)+1 = \dim(\partial G) + 1. \end{equation}
Corollary \ref{dimcor} strengthens \eqref{aicdineq} in exactly the way that Conjecture \ref{conj:dimeq} would predict.

We next recall the definition of a convergence group.
\begin{definition}
Suppose that $M$ is a compact metrizable space with at least $3$
points, and let $G$ act on $M$ by homeomorphisms.  The action is a
\emph{convergence group action} if the induced action on the space
$\Theta^3(M)$ of
unordered triples of distinct points in $M$ is properly discontinuous.

An element $g\in G$ is \emph{loxodromic} if it has infinite order and
fixes exactly two points of $M$.

A point $p\in M$ is a \emph{bounded parabolic point} if $\stab_G(p)$
contains no loxodromics, and acts cocompactly on $M\setminus\{p\}$.

A point $p\in M$ is a \emph{conical limit point} if there is a
sequence $\{g_i\}$ in $G$ and a pair of points $a\neq b$ in $M$ so
that:
\begin{enumerate}
\item $\lim_{i\to \infty} g_i(p) = a$, and
\item $\lim_{i\to\infty} g_i(x) = b$ for all $x\in M\setminus\{p\}$.
\end{enumerate}

A convergence group action of $G$ on $M$ is \emph{geometrically finite} if every
point in $M$ is either a bounded parabolic point or a conical limit point.
\end{definition}

Bowditch proved in \cite{bowditch98} that if $G$ acts on $M$ as a
convergence group and every point of $M$ is a conical limit point,
then $G$ is hyperbolic.  Conversely, if $G$ is hyperbolic, then $G$
acts as a convergence group on $\partial G$, and every point in
$\partial G$ is a conical limit point.
For general geometrically finite actions, we
have the following result of Yaman:

\begin{theorem}\cite[Theorem 0.1]{yaman04}\label{t:yaman}
Suppose that $M$ is a non-empty perfect metrizable compact space, and
suppose that $G$ acts on $M$ as a geometrically finite convergence
group.  Let $B\subset M$ be the set of bounded parabolic points.
Let $\{p_1,\ldots,p_n\}$ be a set of orbit representatives for
the action of $G$ on $B$.  For each $i$ let $P_i$ be the stabilizer in
$G$ of $p_i$, and let $\mc{P} = \{P_1,\ldots,P_n\}$.  

Then $(G,\mc{P})$ is relatively hyperbolic and $M$ is equivariantly homeomorphic to $\partial(G,\mc{P})$.
\end{theorem}

\begin{proof}[Outline of proof of Theorems \ref{t:bowditch} and \ref{t:decomp}]

We prove Theorem \ref{t:bowditch} by constructing a space $M$ on which $G$
acts as a geometrically finite convergence group, so that the
parabolic point stabilizers are all conjugate to elements of $\mc{H}$.
The space $M$ is a quotient of $\partial G$, constructed as follows.
The hypotheses on $\mc{H}$ imply that the boundaries $\partial H_i$
embed in $\partial G$ for each $i$, and that $g \partial H_i \cap h
\partial H_j$ is empty unless $i=j$ and $g^{-1} h \in H_i$.  Let 
\[\mc{A} =
\{g \partial H_i\mid g\in G\mbox{, and }H_i\in \mc{H}\},\]
and let
\[ \mc{B} = \{ \{x\}\mid x\in \partial G\setminus \bigcup A\}. \]
The union $\mc{C} = \mc{A}\cup \mc{B}$ is therefore a decomposition of $\partial
G$ into closed sets.  We let $M$ be the quotient topological space 
$\partial G / \mc{C}$ and write $A,B$ for the images of $\mc{A},\mc{B}$, respectively.
There is clearly an action of $G$ on $M$ by
homeomorphisms.

We now have a sequence of four claims, which we prove in Section \ref{sec:proofs}.
\begin{claim}\label{c:top}
  $M = A\cup B$ is a perfect metrizable space.
\end{claim}

\begin{claim}\label{c:dyn}
  $G$ acts as a convergence group on $M$.
\end{claim}

\begin{claim}\label{c:bpp}
  For $x\in A$, $x$ is a bounded parabolic point, with stabilizer
  conjugate to an element of $\mc{H}$.
\end{claim}

\begin{claim}\label{c:clp}
  For $x\in B$, $x$ is a conical limit point.
\end{claim}
Given the claims, we may apply Yaman's theorem \ref{t:yaman} to
conclude that the pair $(G,\mc{H})$ is relatively hyperbolic (Theorem \ref{t:bowditch}), and that the Bowditch boundary is equivariantly homeomorphic to $\partial G/\mc{C}$ (Theorem \ref{t:decomp}).
\end{proof}

\section{Proofs of claims}\label{sec:proofs}

In what follows we fix some $\delta$-hyperbolic Cayley graph $\Gamma$
of $G$.  We'll use the notation $a\mapsto \bar{a}$ for the map from
$\partial G$ to the decomposition space $M$.

\subsection{Claim \ref{c:top}}  In this subsection we show the decomposition space $M$ is perfect and metrizable.
We first recall the terminology of upper semicontinuous decomposition spaces.
Let $X$ be a topological space, and let $\mc{D}\subseteq 2^X$ be a decomposition of $X$ into compact subsets.  The quotient $X/\mc{D}$ is the \emph{decomposition space}.  As a set $X/\mc{D} = \mc{D}$, topologized so that $A\subset\mc{D}$ is open in $X/\mc{D}$ if and only if $\bigcup A$ is open in $X$.

A subset of $X$ is \emph{$\mc{D}$--saturated} if it is a union of elements of $\mc{D}$.  
  The decomposition is said to be \emph{upper semicontinuous} if for every $D\in \mc{D}$ and every open set $U$ of $X$ containing $D$, there is a $\mc{D}$--saturated open $V$ so that $D\subseteq V\subseteq U$.  

  We have the following well-known characterization.
  \begin{proposition}\label{uscchar}
    Let $X$ be compact metrizable, and let $\mc{D}$ be a decomposition of $X$ into closed subsets.  The following are equivalent:
    \begin{enumerate}
    \item\label{usc} $\mc{D}$ is upper semicontinuous.
    \item\label{cm} $X/\mc{D}$ is compact metrizable.
    \end{enumerate}
  \end{proposition}
  \begin{proof}
    For \eqref{usc}$\implies$\eqref{cm}, see \cite[p. 13, Proposition 2]{DavermanDoM}.  For \eqref{cm}$\implies$\eqref{usc}, see \cite[p. 132, Theorem 3-31]{HY}.
  \end{proof}

  A countable collection of subsets $\mc{N}$ of a metric space $X$ is said to be a \emph{null sequence} if, for all $\epsilon>0$, there are only finitely many $N\in \mc{N}$ of diameter greater than $\epsilon$.  We need the following useful fact.
  \begin{proposition}\cite[p. 14, Proposition 3]{DavermanDoM}\label{nullseq}
    Suppose $X$ is a metric space, and $\mc{D}$ is a decomposition so that the collection of nondegenerate elements of $\mc{D}$ is a null sequence.  Then $\mc{D}$ is upper semicontinuous.
  \end{proposition}

\begin{proof}[Proof of Claim \ref{c:top}]

  We first note that the collection $\mc{A}\subset \mc{C}$ of limit sets of cosets is a null sequence (see \cite[Corollary 2.5]{GMRS98}), and that $\mc{A}$ consists precisely of the nondegenerate elements of $\mc{C}$.   Applying Proposition \ref{nullseq} we see that $\mc{C}$ is an upper semicontinuous decomposition of $\partial{G}$.  
    Proposition \ref{uscchar} then implies that $M = \partial G/\mc{C}$ is compact metric. 
  
  We now show $M$ is perfect.  Let $p\in M$.

  Suppose first
  that $p\in B$, i.e., that the preimage in $\partial G$ is a single
  point $\tilde{p}$.
  Because $G$ is nonelementary,
  $\partial G$ is perfect.  Thus there is a sequence of points $x_i
  \in \partial G\setminus\{\tilde{p}\}$ limiting on $p$.  The image of
  this sequence limits on $p$.

  Now suppose that $p\in A$, i.e., the preimage of $p$ in $\partial G$
  is equal to $g \partial H$ for some $g\in G$ and some $H\in
  \mc{H}$.  Choose any point $x \in \partial G\setminus \partial H$,
  and any infinite order element $h$ of $gHg^{-1}$.  The points $h^i
  x$ project to distinct points in $M\setminus\{p\}$, limiting on $p$.
\end{proof}

\subsection{Claim \ref{c:dyn}}  Next we show the action of $G$ on $M$ is a convergence action.
In \cite{Bo99}, Bowditch gives a characterization of
convergence group actions in terms of \emph{collapsing sets}.  We
rephrase Bowditch slightly in what follows.
\begin{definition}\label{d:collapse}
  Let $G$ act by homeomorphisms on $M$.
  Suppose that $\{g_i\}$ is a sequence of distinct elements of $G$.
  Suppose that there exist points $a$ and $b$ (called the
  \emph{repelling} and \emph{attracting} points, respectively) so that
  whenever $K\subseteq M\setminus \{a\}$ and $L\subseteq
  M\setminus\{b\}$ are compact, the set $\{i\mid g_i K\cap L \neq \emptyset\}$ is
  finite.  Then $\{g_i\}$ is a \emph{collapsing sequence}.
\end{definition}

\begin{proposition}\cite[Proposition 1.1]{Bo99}\label{p:collapse}
  Let $G$, a countable group,
  act on $M$, a compact Hausdorff space with at least $3$ points.
  Then $G$ acts as a convergence group if and only if every infinite
  sequence in $G$ contains a subsequence which is collapsing.
\end{proposition}

\begin{proof}[Proof of Claim \ref{c:dyn}]
  We use the characterization of \ref{p:collapse}.  Let $\{\gamma_i\}$ be
  an infinite sequence in $G$.  Since the action of $G$ on $\partial
  G$ is convergence, there is a collapsing subsequence $\{g_i\}$ of
  $\{\gamma_i\}$; i.e., there are points $a$ and $b$ in $\partial G$
  which are  repelling and attracting in the sense of Definition
  \ref{d:collapse}.   We will show that $\{g_i\}$ is also a collapsing
  sequence for the action of $G$ on $M$, and that the images $\bar{a}$ and
  $\bar{b}$ in $M$ are the repelling and attracting points for this sequence.

  Let $K\subseteq M\setminus \{\bar{a}\}$ and $L\subseteq M\setminus \{\bar{b}\}$
  be compact sets, and let $\tilde{K}$ and $\tilde{L}$ be the
  preimages of $K$ and $L$ in $\partial G$.  We have $\tilde{K}
  \subseteq \partial G\setminus \{a\}$ and $\tilde{L}\subseteq
  \partial G\setminus \{b\}$, so $\{i\mid g_i \tilde{K}\cap
  \tilde{L}\neq\emptyset\}$ is finite.  But for each $i$, $g_i K\cap L
    = \pi(g_i\tilde{K} \cap \tilde{L})$, so $\{i\mid g_i K\cap L\}$ is
    also finite.  
\end{proof}
\begin{remark}
  In the preceding proof it is possible for $a$ and $b$ to be
  distinct, but $\bar{a} = \bar{b}$.
\end{remark}

\subsection{Claim \ref{c:bpp}} We next show the nondegenerate sets in the decomposition give rise to bounded parabolic points.
\begin{proof}[Proof of Claim \ref{c:bpp}]
  Let $p\in A\subseteq M$ be the image of $g \partial H$ for $g\in G$ and $H\in
  \mc{H}$.  Let $P = g H g^{-1}$.  Since $H$ is equal to its own
  commensurator, so is $P$, and $P = \stab_G(p)$.  We must show that
  $P$ acts cocompactly on $M\setminus \{p\}$.    The subgroup
  $P$ is $\lambda$-quasiconvex in $\Gamma$ (the Cayley graph of $G$)
  for some $\lambda>0$.  Let
  $N$ be a closed $R$-neighborhood of $P$ in $\Gamma$ for some large
  integer $R$, with $R > 2\lambda+10\delta$.  Note that any geodesic
  from $1$ to a point in $\partial H$ stays inside $N$, and any
  geodesic from $1$ to a point in $\partial G\setminus \partial P$
  eventually leaves $N$.  Write $\mathrm{Front}(N)$ for the frontier of $N$.

  Let $K =  \{g\in \mathrm{Front}(N)\mid d(g,1)\leq 2 R + 100\delta \}$. Let
  $E$ be the
  set of points $e\in \partial X$ so that there is a geodesic from $1$
  to $e$ passing through $K$.  The set $E$ is compact, and lies
  entirely in $\partial G\setminus\partial P$.  We will show that $PE
  = \partial G\setminus \partial P$.  Let $e\in \partial
  G\setminus\partial P$, and let $h\in P$ be ``coarsely closest'' to $e$ in the
  following sense:  If $\{x_i\}$ is a sequence of points in $X$
  tending to $e$, then for large enough $i$, we have, for any $h'\in
  P$, $d(h,e_i)\leq d(h',e_i)+4\delta$.

  Let $\gamma$ be a geodesic ray from $h$ to $e$, and let
  $d$ be the unique point in $\gamma\cap \mathrm{Front}(N)$.  Since $d\in
  \mathrm{Front}(N)$, there is some
  $h'$ so that $d(h',d) = R$.  Let $e'$ be a point on $\gamma$ so that
  $10R<d(h,e')\leq d(h',e')+4\delta$, and consider a geodesic triangle
  made up of that part of $\gamma$ between $h$ and $e'$, some geodesic
  between $h'$ and $h$, and some geodesic between $h'$ and $e'$.  This
  triangle has a corresponding comparison tripod, as in Figure
  \begin{figure}
    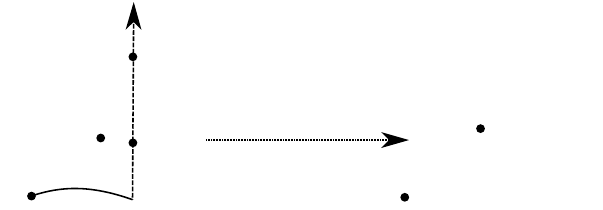
    \caption{Bounding the distance from $h$ to $d$.}
    \label{tripod}
  \end{figure}
  \ref{tripod}.
  Since any geodesic from $h'$ to $h$ must stay $R-\lambda>\delta$ away from
  $\mathrm{Front}(N)$, the point $\bar{d}$ must lie on the leg of the tripod
  corresponding to $e'$.  Let $d'$ be the point on the geodesic from
  $h'$ to $e'$ which projects to $\bar{d}$ in the comparison tripod.
  Since $d(h',d) = R$, $d(h',d')\leq R+\delta$.  Now notice that 
  \begin{eqnarray*} 
    d(h, d) & \leq & d(h',d') + (e',h')_h - (e',h)_{h'}\\
    & \leq & d(h',d')+4\delta\\
    & \leq & R+5\delta.
  \end{eqnarray*}
  But this implies that the geodesic from $1$ to $h^{-1}e$ passes
  through $K$, and so $h^{-1}e \in E$ and $e\in hE$.  Since $e$ was arbitrary in
  $\partial G\setminus \partial P$, we have $PE = \partial
  G\setminus\partial P$, and so the action of $P$ on $\partial
  G\setminus \partial P$ is cocompact.  If $\bar{E}$ is the (compact)
  image of $E$ in $M$, then $PE = M\setminus\{p\}$, and so $p$ is a
  bounded parabolic point.
\end{proof}

\subsection{Claim \ref{c:clp}} Lastly, we must show that the remaining points of $M$ are conical limit points. 
The proofs of the first two lemmas are left to the reader.
\begin{lemma}\label{l:bci}
  For all $R>0$ there is some $D$, depending only on $R$, $G$, $\mc{H}$,
  and $S$, so that for any $g$, $g'\in G$, and $H$, $H'\in \mc{H}$, 
  \begin{equation*}
    \mathrm{diam}(N_R(gH)\cap N_R(g'H'))<D.
  \end{equation*}
  ($N_R(Z)$ denotes the $R$-neighborhood of $Z$ in the Cayley graph $\Gamma = \Gamma(G,S)$.)
\end{lemma}

\begin{lemma}\label{l:lambda}
  There is some $\lambda$ depending only on $G$, $\mc{H}$, and $S$, so
  that if $x$, $y\in gH\cup g\partial H$, then any geodesic from $x$
  to $y$ lies in a $\lambda$-neighborhood of $gH$ in $\Gamma$.
\end{lemma}

\begin{lemma}\label{l:sequence}
  Let $\gamma\co\R_+\to \Gamma$ be a (unit speed) geodesic ray, so
  that $x=\lim_{t\to\infty}\gamma(t)$ is not in the limit set of $gH$
  for any $g\in G$, $H\in \mc{H}$, and so that $\gamma(0)\in G$.  Let $C>0$.
  There is a sequence of numbers $\{n_i\}$ tending to infinity, and a
  constant $\chi$, so that the following holds, for all $i\in \N$:
  If $x_i = \gamma(n_i) \in N_C(gH)$ for $g\in G$ and $H\in \mc{H}$,
  then 
  \begin{equation*}
    \mathrm{diam}\left(N_C(gH)\cap \gamma([n_i,\infty))\right)<\chi
  \end{equation*}
\end{lemma}
\begin{proof}
  Let $\lambda$ be the quasi-convexity constant from Lemma \ref{l:lambda}.
  Let $D$ be the constant obtained from Lemma \ref{l:bci}, setting
  $R=C+\lambda+2\delta$, and let $\chi=2 D$. 

  We define $n_i$ inductively.  
  Let $i\in \N$.  If $i = 1$, set $t_1=0$;
  otherwise set $t_i = n_{i-1}+1$.  We will find $n_i\geq t_i$ satisfying
  the condition in the statement.

  If setting $n_i=t_i$ does not work, then there must be some $gH$ with $g\in
  G$ and $H\in \mc{H}$ satisfying $\gamma(t_i)\in N_C(gH)$ and 
  \begin{equation*}
    \mathrm{diam}\left(N_C(gH)\cap \gamma([t_i,\infty))\right)\geq \chi.
  \end{equation*}
  Let $s = \sup\{t\mid \gamma(t)\in N_C(gH)\}$.    We claim
  that we can choose 
  \[n_i = s - \frac{\chi}{2} = s - D.\]
  Clearly we have
  \begin{equation*}
    \mathrm{diam}\left(N_C(gH)\cap \gamma([n_i,\infty))\right)<\chi.
  \end{equation*}
  Now suppose that some other $g'H'$ satisfies
  $x_i = \gamma(n_i)\in  N_C(g'H')$ and 
  \begin{equation*}
    \mathrm{diam}\left(N_C(g'H')\cap \gamma([n_i,\infty))\right)\geq \chi.
  \end{equation*}
  There is then some $\Delta\ge 0$ so that $\gamma(s+D+\Delta)$ is within $C$ of $g'H'$.
  It is straightforward to show (see Figure
  \begin{figure}
    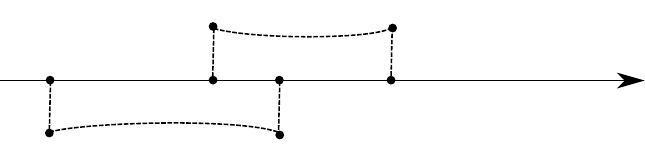
    \caption{The segment $\gamma|_{[n_i,s]}$ is close to both $gH$ and $g'H'$.}
    \label{quads}
  \end{figure}
  \ref{quads}) that $\gamma(n_i)$ and
  $\gamma(s)$ lie both in the $C+\lambda+2\delta$ neighborhood of
  $gH$ and in the $C+\lambda+2\delta$ neighborhood of $g'H'$.  Since
  $d(\gamma(n_i),\gamma(s))=s-n_i=D$, this contradicts Lemma \ref{l:bci}.
\end{proof}
\begin{proof}[Proof of Claim \ref{c:clp}]
  Let $x\in B= \partial G\setminus \cup \mc{A}$.  We must show that
  $\bar{x}\in M$ is a conical limit point for the action of $G$ on
  $M$.  Fix some $y\in M\setminus\{x\}$, and let $\gamma$ be a
  geodesic from $y$ to $x$ in $\Gamma$.  Let $C = \lambda+6\delta$,
  where $\lambda$ is the constant from Lemma \ref{l:lambda}.  Using
  Lemma \ref{l:sequence}, we can choose a sequence of (inverses of)
  group elements $\{{x_i}^{-1}\}$ in the image of $\gamma$ so that
  whenever $x_i\in N_C(gH)$ for some $g\in G$, $H\in \mc{H}$, and
  $i\in \N$, we have 
  \begin{equation}\label{diambound}
    \mathrm{diam}\left(N_C(gH)\cap \gamma([n_i,\infty))\right)<\chi,
  \end{equation}
  for some constant $\chi$ independent of $g$, $H$, and $i$.

  Now consider the geodesics $x_i\gamma$.  They all pass through $1$,
  so we may pick a subsequence $\{x_i'\}$ so that the geodesics
  $x_i'\gamma$ converge setwise to a geodesic $\sigma$ running from
  $b$ to $a$ for some $b$, $a\in \partial G$.  In fact this sequence
  $\{x_i'\}$ will satisfy $\lim_{i\to\infty}x_i'x = a$ and
  $\lim_{i\to\infty}x_i' y' = b$ for all $y'\in \partial G\setminus
  \{x\}$.  We will be able to use this sequence to see that $\bar{x}$
  is a conical limit point for the action of $G$ on $M$, \emph{unless}
  we have $\bar{a} = \bar{b}$ in $M$.

  By way of contradiction, we therefore assume that $a$ and $b$ both
  lie in $g\partial H$ for some $g\in G$, and $H\in \mc{H}$.  The
  geodesic $\sigma$ lies in a $\lambda$-neighborhood of $gH$, by Lemma
  \ref{l:lambda}.  Let $R>\chi$, and let $B_R(1)$ be the $R$--ball around the identity in the Cayley graph $\Gamma$.  The set $x_i\gamma \cap B_R(1)$ must
  eventually be constant, equal to $\sigma_R :=\sigma\cap B_R(1)$.
  Now $\sigma_R$ a geodesic segment of length $2R$ lying entirely
  inside $N_C(gH)$.  It follows that, for sufficiently large $i$,
  ${x_i'}^{-1}\sigma_R\subseteq \gamma$ lies inside
  $N_C({x_i'}^{-1}gH)$.  In particular, if ${x_i'}^{-1} = \gamma(t_i)$,
  then we have $\gamma([t_i,t_i+R))\subseteq N_C({x_i'}^{-1}gH)$.
  $R>\chi$, this contradicts \eqref{diambound}.  
\end{proof}

\section{Acknowledgments} Thanks to Saul Schleimer and the referee for helpful corrections.

\end{document}

%% file: trip.pdf_tex
%% Accompanies image file 'trip.pdf' (pdf, eps, ps)
\begingroup%
  \makeatletter%
  \providecommand\color[2][]{%
    \errmessage{(Inkscape) Color is used for the text in Inkscape, but the package 'color.sty' is not loaded}%
    \renewcommand\color[2][]{}%
  }%
  \providecommand\transparent[1]{%
    \errmessage{(Inkscape) Transparency is used (non-zero) for the text in Inkscape, but the package 'transparent.sty' is not loaded}%
    \renewcommand\transparent[1]{}%
  }%
  \providecommand\rotatebox[2]{#2}%
  \newcommand*\fsize{\dimexpr\f@size pt\relax}%
  \newcommand*\lineheight[1]{\fontsize{\fsize}{#1\fsize}\selectfont}%
  \ifx\svgwidth\undefined%
    \setlength{\unitlength}{289.1950065bp}%
    \ifx\svgscale\undefined%
      \relax%
    \else%
      \setlength{\unitlength}{\unitlength * \real{\svgscale}}%
    \fi%
  \else%
    \setlength{\unitlength}{\svgwidth}%
  \fi%
  \global\let\svgwidth\undefined%
  \global\let\svgscale\undefined%
  \makeatother%
  \begin{picture}(1,0.34715285)%
    \lineheight{1}%
    \setlength\tabcolsep{0pt}%
    \put(0,0){\includegraphics[width=\unitlength,page=1]{trip.pdf}}%
    \put(0.10891561,0.12363357){\color[rgb]{0,0,0}\makebox(0,0)[lt]{\lineheight{1.25}\smash{\begin{tabular}[t]{l}$d'$\end{tabular}}}}%
    \put(0.23429944,0.10201517){\color[rgb]{0,0,0}\makebox(0,0)[lt]{\lineheight{1.25}\smash{\begin{tabular}[t]{l}$d$\end{tabular}}}}%
    \put(-0.00229624,0.01209738){\color[rgb]{0,0,0}\makebox(0,0)[lt]{\lineheight{1.25}\smash{\begin{tabular}[t]{l}$h'$\end{tabular}}}}%
    \put(0.23160675,0.0040657){\color[rgb]{0,0,0}\makebox(0,0)[lt]{\lineheight{1.25}\smash{\begin{tabular}[t]{l}$h$\end{tabular}}}}%
    \put(0,0){\includegraphics[width=\unitlength,page=2]{trip.pdf}}%
    \put(0.18214433,0.24416534){\color[rgb]{0,0,0}\makebox(0,0)[lt]{\lineheight{1.25}\smash{\begin{tabular}[t]{l}$e'$\end{tabular}}}}%
    \put(0.23818041,0.32613546){\color[rgb]{0,0,0}\makebox(0,0)[lt]{\lineheight{1.25}\smash{\begin{tabular}[t]{l}$e$\end{tabular}}}}%
    \put(0,0){\includegraphics[width=\unitlength,page=3]{trip.pdf}}%
    \put(0.85194624,0.23780321){\color[rgb]{0,0,0}\makebox(0,0)[lt]{\lineheight{1.25}\smash{\begin{tabular}[t]{l}$\bar{e}'$\end{tabular}}}}%
    \put(0.81265016,0.12122499){\color[rgb]{0,0,0}\makebox(0,0)[lt]{\lineheight{1.25}\smash{\begin{tabular}[t]{l}$\bar{d}$\end{tabular}}}}%
    \put(0.84539697,0.02298492){\color[rgb]{0,0,0}\makebox(0,0)[lt]{\lineheight{1.25}\smash{\begin{tabular}[t]{l}$\bar{h}$\end{tabular}}}}%
    \put(0.61224037,0.01643557){\color[rgb]{0,0,0}\makebox(0,0)[lt]{\lineheight{1.25}\smash{\begin{tabular}[t]{l}$\bar{h}'$\end{tabular}}}}%
  \end{picture}%
\endgroup%

%% file: quads.pdf_tex
%% Accompanies image file 'quads.pdf' (pdf, eps, ps)
\begingroup%
  \makeatletter%
  \providecommand\color[2][]{%
    \errmessage{(Inkscape) Color is used for the text in Inkscape, but the package 'color.sty' is not loaded}%
    \renewcommand\color[2][]{}%
  }%
  \providecommand\transparent[1]{%
    \errmessage{(Inkscape) Transparency is used (non-zero) for the text in Inkscape, but the package 'transparent.sty' is not loaded}%
    \renewcommand\transparent[1]{}%
  }%
  \providecommand\rotatebox[2]{#2}%
  \newcommand*\fsize{\dimexpr\f@size pt\relax}%
  \newcommand*\lineheight[1]{\fontsize{\fsize}{#1\fsize}\selectfont}%
  \ifx\svgwidth\undefined%
    \setlength{\unitlength}{309.42858072bp}%
    \ifx\svgscale\undefined%
      \relax%
    \else%
      \setlength{\unitlength}{\unitlength * \real{\svgscale}}%
    \fi%
  \else%
    \setlength{\unitlength}{\svgwidth}%
  \fi%
  \global\let\svgwidth\undefined%
  \global\let\svgscale\undefined%
  \makeatother%
  \begin{picture}(1,0.25099613)%
    \lineheight{1}%
    \setlength\tabcolsep{0pt}%
    \put(0,0){\includegraphics[width=\unitlength,page=1]{quads.pdf}}%
    \put(0.07222872,0.15300326){\color[rgb]{0,0,0}\makebox(0,0)[lt]{\lineheight{1.25}\smash{\begin{tabular}[t]{l}$\gamma(t_i)$\end{tabular}}}}%
    \put(0.42357854,0.14688217){\color[rgb]{0,0,0}\makebox(0,0)[lt]{\lineheight{1.25}\smash{\begin{tabular}[t]{l}$\gamma(s)$\end{tabular}}}}%
    \put(0.29870852,0.09056832){\color[rgb]{0,0,0}\makebox(0,0)[lt]{\lineheight{1.25}\smash{\begin{tabular}[t]{l}$\gamma(n_i)$\end{tabular}}}}%
    \put(0.60476242,0.08689569){\color[rgb]{0,0,0}\makebox(0,0)[lt]{\lineheight{1.25}\smash{\begin{tabular}[t]{l}$\gamma(s+D+\Delta)$\end{tabular}}}}%
    \put(0.06855606,0.00609743){\color[rgb]{0,0,0}\makebox(0,0)[lt]{\lineheight{1.25}\smash{\begin{tabular}[t]{l}$gh_1$\end{tabular}}}}%
    \put(0.42480263,0.00854589){\color[rgb]{0,0,0}\makebox(0,0)[lt]{\lineheight{1.25}\smash{\begin{tabular}[t]{l}$gh_2$\end{tabular}}}}%
    \put(0.32441709,0.23135306){\color[rgb]{0,0,0}\makebox(0,0)[lt]{\lineheight{1.25}\smash{\begin{tabular}[t]{l}$g'h_3$\end{tabular}}}}%
    \put(0.60353819,0.23012883){\color[rgb]{0,0,0}\makebox(0,0)[lt]{\lineheight{1.25}\smash{\begin{tabular}[t]{l}$g'h_4$\end{tabular}}}}%
  \end{picture}%
\endgroup%